%% file: main.tex
\documentclass[11pt]{amsart}
\usepackage{amssymb}
\usepackage{ytableau}
\usepackage{mathtools,color}
\usepackage{amsrefs}
\usepackage{tikz}
\usetikzlibrary{arrows.meta}
\usepackage{fullpage}
\theoremstyle{plain}
\newtheorem{theorem}{Theorem}[section]
\newtheorem{prop}[theorem]{Proposition}
\newtheorem{lemma}[theorem]{Lemma}
\newtheorem{corollary}[theorem]{Corollary}

\theoremstyle{definition}

\newtheorem{rem}[theorem]{Remark}

\newtheorem{ex}[theorem]{Example}

\numberwithin{equation}{section}

\newcommand{\C}{\mathcal{C}}
\newcommand{\EMD}{{\rm EMD}}
\newcommand{\Par}{{\rm Par}}
\newcommand{\smtriangle}{\scalebox{.8}{$\,\triangle\,$}}
\newcommand{\Stri}{{\rm S}_{_{\!\triangle\!\!}}}

\begin{document}

\title{
Palindromicity of the numerator \\ of a statistical generating function}

\author{Rebecca Bourn}
\address{
Rebecca Bourn\\
Department of Mathematical Sciences\\
University of Wisconsin--Milwaukee \\ 
3200 N.~Cramer St.\\
Milwaukee, WI 53211} 
\email{bourn@uwm.edu}

\author{William Q. Erickson}
\address{
William Q.~Erickson\\
Department of Mathematics\\
Baylor University \\ 
One Bear Place \#97328\\
Waco, TX 76798} 
\email{will\_erickson@baylor.edu}

\begin{abstract}
    We prove a conjecture of Bourn and Willenbring (2020) regarding the palindromicity and unimodality of a certain family of polynomials $N_n(t)$.
    These recursively defined polynomials arise as the numerators of generating functions in the context of the discrete one-dimensional earth mover's distance (EMD).
    The key to our proof is showing that the defining recursion can be viewed as describing sums of symmetric differences of pairs of Young diagrams; in this setting, palindromicity is equivalent to the preservation of the symmetric difference under the transposition of diagrams.
    We also observe a connection to recent work by Defant et al.~(2024) on the Wiener index of minuscule lattices, which we reinterpret combinatorially to obtain explicit formulas for the coefficients of $N_n(t)$ and for the expected value of the discrete EMD.
    
\end{abstract}



\maketitle

\section{Introduction}

\subsection{Background}

In this note, we prove a conjecture on the palindromicity and unimodality of the numerators of a certain family of statistical generating functions, by reinterpreting them in terms of Young diagrams.

In \cite{BW}, Bourn and Willenbring derive a recursive formula for the expected value of the one-dimensional earth mover's distance (EMD).
For the purposes of this note, it is enough to know that the EMD (also called the first Wasserstein distance) measures the distance between two histograms (or more generally, probability distributions), by computing the minimum amount of ``work'' required to transform one into the other (see~\cite{Rubner}).
The EMD can also be viewed as the solution to the classical \emph{transportation problem} (often named after various combinations of Hitchcock, Monge, Kantorovich, and Koopmans), and has an ever-widening range of important applications in mathematics along with the physical and social sciences.
We recommend Villani's monumental reference~\cite{Villani} for further reading.

While the main result in~\cite{BW} extends to probability distributions on the set $[n] \coloneqq \{1, \ldots, n\}$, the authors first consider discrete histograms with $n$ bins and $s$ data points --- in other words, the set of (weak) integer \emph{compositions} of $s$ into $n$ parts, denoted by
\[
    \C(s,n) \coloneqq \Big\{ (\alpha_1, \ldots, \alpha_n) \in (\mathbb{Z}_{\geq 0})^n \; \Big| \; \textstyle\sum_i \alpha_i = s \Big\}.
\]
Then, temporarily allowing ordered pairs of compositions with different numbers of parts ($p$ and $q$), they define a bivariate generating function
\begin{equation}
    \label{Hpq}
    H_{pq}(z,t) \coloneqq \sum_{s=0}^\infty \Bigg(\sum_{\substack{\alpha \in \C(s,p), \\ \beta \in \C(s,q)\phantom{,}}} z^{\EMD(\alpha,\beta)}\Bigg)t^s
\end{equation}
which tracks the EMD values while summing over all possible numbers $s$ of data points.
Upon differentiating with respect to $z$ and evaluating at $z=1$, they obtain another generating function in $t$, which is the subject of the present note:
\begin{align}
    H'_{pq}(t) &\coloneqq \frac{\partial}{\partial z} H_{pq}(z,t)\bigg|_{z=1} \nonumber \\
    &= \sum_{s=0}^\infty \Bigg(\sum_{\substack{\alpha \in \C(s,p), \\ \beta \in \C(s,q)\phantom{,}}} \hspace{-2ex} \EMD(\alpha,\beta)\Bigg) t^s. \label{H prime}
\end{align}
It is then shown~\cite{BW}*{Prop.~6} that this generating function has the rational form
\[
    H'_{pq}(t) = \frac{N_{pq}(t)}{(1-t)^{p+q}},
\]
where $N_{pq}(t)$ is a recursively defined polynomial (see Section~\ref{sub:Npq recursion} below).

\subsection{Main result}

In practice, one is interested only in the special case $p=q$, in which case we replace both parameters by $n$, and write $H'_n(t) \coloneqq H'_{nn}(t)$ and $N_{n}(t) \coloneqq N_{nn}(t)$.
The main result in this note is a proof of the following theorem, which was conjectured in~\cite{BW}*{Conj.~1}:

\begin{theorem}
\label{thm:main conjecture}
For all positive integers $n$, the polynomial $N_n(t)$ is palindromic and unimodal.
\end{theorem}

Given a polynomial $f(t) = \sum_{k=a}^b f_k t^k$ with total degree $d \coloneqq a+b$ (assuming $f_a \neq 0$ and $f_b \neq 0$), we say that $f(t)$ is \emph{palindromic} if $f_k = f_{d-k}$ for all $a \leq k \leq b$.
This is equivalent to the condition $f(t) = t^d f(1/t)$.
Moreover, $f(t)$ is said to be \emph{unimodal} if
\[
f_a \leq f_{a+1} \leq \cdots \leq f_m \geq \cdots \geq f_{b-1} \geq f_b
\]
for some index $m$.
The key to our proof of palindromicity of the polynomials $N_n(t)$ is a combinatorial interpretation of the recursion defining $N_{pq}(t)$, in terms of the symmetric difference of pairs of Young diagrams bounded by certain rectangles (see Theorem~\ref{thm:main result}).
The palindromic property of $N_n(t)$ then becomes clear from the fact that the size of the symmetric difference is unchanged by reflecting both diagrams (see Corollary~\ref{cor:main}).

\subsection{Connections to other work}

In Section~\ref{sec:Defant} we point out a surprising connection to a recent paper by Defant et al.~\cite{Defant}.
In particular, the distance defined on a minuscule lattice of Type A measures the symmetric difference between Young diagrams, and the \emph{Wiener index} is the sum of all these pairwise distances.
By aligning this fact with our combinatorial view in Theorem~\ref{thm:main result}, we realize that the Wiener index formula in~\cite{Defant} yields the coefficients of $N_n(t)$, giving an explicit description of these once-mysterious polynomials (see Theorem~\ref{thm:coeffs of Nn}).
This leads to a straightforward proof (Corollary~\ref{cor:unimodal}) of the unimodality of $N_n(t)$, thereby completing the proof of Theorem~\ref{thm:main conjecture}.

As it turns out, the Wiener index in~\cite{Defant} gives the explicit coefficient formula not only for the numerator $N_n(t)$, but also for the series expansion $H'_{n}(t)$.
This is sufficient to close the recursion in~\cite{BW} for the expected value of the one-dimensional EMD on $\C(s,n) \times \C(s,n)$, and the closed form is surprisingly simple (see Theorem~\ref{thm:EMD exp val}):
\begin{equation}
\label{EMD exp val}
    \mathbb{E}\Big[\EMD(\alpha,\beta)\Big] = \frac{s(n-1)}{4s+4n-2} \cdot \frac{\binom{2s+2n}{2s+1}}{\binom{s+n-1}{s}^2}.
\end{equation}
We emphasize that this result complements two similar extensions of the problem from the original paper~\cite{BW}.
First, the authors of~\cite{FV} derive a non-recursive formula for the expected value of the EMD, but on pairs of probability distributions rather than discrete histograms.  (See our Remark~\ref{rem FV}, where we recover this result by taking a single limit.)
Second, the authors of~\cite{EK23} present a non-recursive approach in the discrete setting, which is less direct than~\eqref{EMD exp val} but also more flexible.
Specifically, their formula gives the entries of a certain matrix, which (upon taking the trace of its product with an arbitrary Monge cost matrix $C$) yields the expected value of the modified EMD determined by $C$.

\subsection*{Acknowledgments}

We thank the anonymous referees for their careful reading and helpful comments on the manuscript.

\section{Combinatorial preliminaries}

\subsection{Recursive definition of $N_{pq}(t)$}
\label{sub:Npq recursion}

Recall from~\eqref{Hpq} the generating function $H_{pq}(z,t)$, and consider the specialization evaluated at $z=1$:
\begin{equation}
    \label{Hpq(1)}
    H_{pq}(1,t) = \sum_{s = 0}^\infty |\C(s,p)| \cdot |\C(s,q)| \cdot t^s.
\end{equation}
In~\cite{BW} this is rewritten in the rational form
\[
    H_{pq}(1,t) = \frac{W_{pq}(t)}{(1-t)^{p+q-1}},
\]
where the coefficients of the numerator are well known to be given by
    \begin{equation}
    \label{Wpq}
    W_{pq}(t) = \sum_{k=0}^{\min\{p,q\}-1} \binom{p-1}{k} \binom{q-1}{k} t^k.
    \end{equation}
By first proving a recursion for $H_{pq}(t)$ and then manipulating generating functions, the authors of~\cite{BW} derive the following recursive definition of $N_{pq} \coloneqq N_{pq}(t)$ in their equation (4.3):
\begin{equation}
    \label{Npq recursion}
    N_{pq} = N_{p-1,q} + N_{p,q-1} - (1-t)N_{p-1,q-1} + |p-q| \cdot t \, W_{pq},
\end{equation}
with initial values $N_{0,q} = N_{p,0} = N_{1,1} = 0$.
Below we list the polynomials $N_n(t) \coloneqq N_{nn}(t)$ for the first few values of $n$:
\begin{align*}
    N_1(t) &= 0,\\
    N_2(t) &= 2t,\\
    N_3(t) &= 8t + 8t^2,\\
    N_4(t) &= 20t + 56t^2 + 20t^3,\\
    N_5(t) &= 40t + 216t^2 + 216t^3 + 40t^4,\\
    N_6(t) &= 70 t + 616 t^2 + 1188 t^3 + 616 t^4 + 70 t^5, \\
    N_7(t) &= 112 t + 1456 t^2 + 4576 t^3 + 4576 t^4 + 1456 t^5 + 112 t^6,\\
    N_8(t) &= 168 t + 3024 t^2 + 14040 t^3 + 22880 t^4 + 14040 t^5 + 3024 t^6 + 
 168 t^7.
\end{align*}

\begin{rem}
In~\cite{EW04}*{eqn.~\!(1.4.4)}, the series $H_{pq}(1,t)$ is shown to be the Hilbert series of the first Wallach representation of the Lie algebra $\mathfrak{su}(p,q)$.
It can also be viewed as the Hilbert series of the determinantal variety consisting of $p \times q$ matrices of rank at most 1; or again, as the Hilbert series of the Stanley--Reisner ring of the order complex $\Delta_{pq}$ on the poset $[p] \times [q]$.
In this last context, the coefficient of $t^k$ in the numerator $W_{pq}(t)$ equals the number of facets of $\Delta_{pq}$ whose restrictions have cardinality $k$; equivalently, this is the number of lattice paths from $(1,1)$ to $(p,q)$ containing exactly $k$ ``turns'' from the north to the east.
This is a special case of lattice path enumeration as detailed in~\cite{Krattenthaler}*{Prop.~28 and Fig.~8}, for example.
See also~\cite{Billera}, which introduced these methods in the context of the harmonic oscillator.
\end{rem}

\subsection{Symmetric difference of partitions}

A \emph{partition} $\lambda = (\lambda_1, \ldots, \lambda_r)$ is a weakly decreasing sequence of positive integers.
We call $r$ the \emph{length} of $\lambda$, and we call $\lambda_1$ the \emph{width} of $\lambda$.
These terms are quite natural when we identify a partition $\lambda$ with its \emph{Young diagram}: the array of boxes, justified along the left edge, in which the row lengths from top to bottom are $\lambda_1, \ldots, \lambda_r$.
We write $\lambda'$ to denote the \emph{conjugate} partition obtained by reflecting the Young diagram $\lambda$ across the main diagonal.
For example, if $\lambda = (6,5,2,2,1)$, then we have
\[
    \ytableausetup{centertableaux,boxsize=6pt}
    \lambda = (6,5,2,2,1) = \ydiagram{6,5,2,2,1}, \qquad 
    \lambda' = (5,4,2,2,2,1) = \ydiagram{5,4,2,2,2,1},
\]
where $\lambda$ has length $5$ and width $6$, while $\lambda'$ has length $6$ and width $5$.

For nonnegative integers $a,b$, let 
\[
\Par(a \times b) \coloneqq \Big\{ \text{partitions $\lambda$} \; \Big| \; \text{$\lambda$ has length $\leq a$ and width $\leq b$}\Big\}.
\]
It is convenient to regard $\Par(a \times b)$ as the set of Young diagrams fitting inside a rectangle with $a$ rows and $b$ columns.
The number of such partitions is well known~\cite{Stanley}*{Prop.~1.2.1} to be given by the binomial coefficient
\begin{equation}
    \label{size Par ab}
    |\Par(a \times b)| = \binom{a+b}{a}.
\end{equation}
Clearly $\lambda \mapsto \lambda'$ is a bijection between $\Par(a \times b)$ and $\Par(b \times a)$.
Note that $\Par(0 \times b) = \Par(a \times 0) = \{\varnothing\}$, where $\varnothing$ is the empty Young diagram.

Given two Young diagrams $\lambda, \mu$, their \emph{symmetric difference} is
\[
    \lambda \smtriangle \mu \coloneqq (\lambda \cup \mu) \setminus (\lambda \cap \mu),
\]
where as usual $\lambda \cup \mu$ denotes the set of boxes occurring in either diagram, and $\lambda \cap \mu$ denotes the set of boxes occurring in both diagrams.
Hence $\lambda \smtriangle \mu$ is the set of boxes which occur in exactly one of the two diagrams.
For example,
\[
\text{if } \lambda = \ydiagram[*(black!15)]{6,5,2,2,1} \text{ and } \mu = \ydiagram[*(black!40)]{4,4,4,3,0}, \text{ then } \lambda \smtriangle \mu = \ydiagram[*(black!15)]{4+2,4+1,0,0,1}*[*(black!40)]{0,0,2+2,2+1}*
{6,5,4,3,1},
\]
and so $|\lambda \smtriangle \mu| = 7$, visualized as the seven shaded cells on the right-hand side.
Crucial to our result will be the following sum of cardinalities of symmetric differences, which we denote by
\begin{equation}
\label{S triangle}
    \Stri(a,b \mid c,d) \coloneqq \sum_{\mathclap{\substack{\lambda \in \Par(a \times b), \\ \mu \in \Par(c \times d)\phantom{,}}}} |\lambda \smtriangle \mu|.
\end{equation}

\begin{lemma}
    \label{lemma:conjugate}
    We have $\Stri(a, b \mid c,d) = \Stri(b,a \mid d,c)$.
\end{lemma}

\begin{proof}
    Recalling the bijection $\lambda \mapsto \lambda'$, we see that $\Stri(b,a \mid d,c)$ is obtained from~\eqref{S triangle} by replacing $\lambda$ with $\lambda'$ and $\mu$ with $\mu'$.
    Since $\lambda' \smtriangle \mu'$ is just the set of boxes obtained by reflecting the set $\lambda \smtriangle \mu$ across the main diagonal, we have $|\lambda \smtriangle \mu| = |\lambda' \smtriangle \mu'|$.
    \end{proof}

\begin{lemma}
    \label{lemma:construction}
    For positive integers $k, \ell, m$, we have
    \begin{align}
    \label{lemma equation}
    \begin{split}
    \Stri(k, \ell \mid k, m) &= \Stri(k, \ell-1 \mid k, m) + \Stri(k, \ell \mid k, m-1) - \Stri(k, \ell-1 \mid k, m-1) \\
    & \hspace{10pt} + \Stri(k-1, \ell \mid k-1, m) \\
    & \hspace{+20pt} + |\ell - m| \cdot |\Par((k-1) \times \ell)|\cdot |\Par((k-1) \times m)|.
    \end{split}
    \end{align}
\end{lemma}

\begin{proof}
    This can be seen by the inclusion--exclusion principle, as follows.
    We claim that
    \begin{equation}
    \label{first line}
        \sum_{\mathclap{\substack{(\lambda,\mu) \\ \in \: \Par(k \times \ell) \times \Par(k \times m) \setminus Q}}}|\lambda \smtriangle \mu| = \text{first line of the right-hand side of~\eqref{lemma equation},}    \end{equation}
    where $Q$ is the subset containing those pairs $(\lambda,\mu)$ such that $\lambda$ has $\ell$ nonempty columns and $\mu$ has $m$ nonempty columns.
    It is clear that any pair $(\lambda,\mu) \in Q$ is not counted in the first line of~\eqref{lemma equation}; to prove the claim~\eqref{first line}, we observe that each pair $(\lambda,\mu) \in \Par(k \times (\ell-1)) \times \Par(k \times (m-1))$ is counted in \emph{both} of the first two terms, which we correct by subtracting the third term.

    Now we claim that the last two lines on the right-hand side of ~\eqref{lemma equation} equal
    \begin{equation}
        \label{last two}
        \sum_{\mathclap{(\lambda,\mu) \in Q}}|\lambda \smtriangle \mu|.
    \end{equation}
    We start with the third line.
    Note that $(\lambda,\mu) \in Q$ if and only if $\lambda$ has width $\ell$ and $\mu$ has width $m$.
    Hence each pair $(\lambda, \mu) \in Q$, if we restrict our attention to the top row of each diagram, contributes $|\ell - m|$ to~\eqref{last two}.
    Moreover, we have
    \[
    |Q| = |\Par((k-1) \times \ell)|\cdot |\Par((k-1) \times m)|,
    \]
    since each $(\lambda, \mu) \in Q$ is determined by the unique pair $(\hat{\lambda}, \hat{\mu})$ where 
    \begin{align*}
        \hat{\lambda} &\in \Par((k-1) \times \ell) \text{ is the complement of the top row in $\lambda$,}\\
        \hat{\mu} &\in \Par((k-1) \times m) \text{ is the complement of the top row in $\mu$.}
    \end{align*}
    Hence we have
    \begin{equation}
        \label{last claim}
        \sum_{\mathclap{(\lambda,\mu) \in Q}}|\lambda \smtriangle \mu| = \underbrace{|\ell - m| \cdot |Q|}_{\text{third line of~\eqref{lemma equation}}} + \underbrace{\sum_{(\lambda, \mu) \in Q} \hspace{-1ex}|\hat{\lambda} \smtriangle \hat{\mu}|.}_{\text{second line of~\eqref{lemma equation}}}
    \end{equation}
    The result~\eqref{lemma equation} follows, upon comparing~\eqref{S triangle} with~\eqref{first line} and~\eqref{last claim}.
\end{proof}

\begin{rem}
    The proof of Lemma~\ref{lemma:construction} can be conveniently visualized as follows.
    Define the symbols
    \[
     \input{S_def.tex} \coloneqq \Stri(k, \ell \mid k, m), \qquad \input{L_def}
    \]
    where we have arbitrarily depicted $\ell \leq m$.
    Then the recursive identity~\eqref{lemma equation} can be expressed in terms of these symbols, where a missing strip denotes exactly one row or column:
    \begin{align*}
        \input{S_LHS} =& && \input{S_term1} + \input{S_term2} - \input{S_term3} \\
        & + &&\input{S_term4} \\
        & + && \input{L_term5} \; \cdot \input{Par1_term5} \cdot \input{Par2_term5}\!.
    \end{align*}
\end{rem}

\section{Main result}
\label{sec:main result}

\begin{theorem}
    \label{thm:main result}
    Let $N_{pq}(t)$ be the polynomial defined recursively in~\eqref{Npq recursion}.
    Then
    \begin{equation}
    \label{Npq in theorem}
        N_{pq}(t) = \sum_{k=1}^{\min\{p,q\}} \Stri(k, p-k \mid k, q-k) \cdot t^k,
    \end{equation}
    where the coefficients $\Stri( - )$ are defined in~\eqref{S triangle}.
\end{theorem}

\begin{proof}
    Recall from~\eqref{Npq recursion} the recursion defining $N_{pq}(t)$, which we expand below (labeling each term for clarity later in the proof):
    \begin{equation}
        \label{Npq expanded}
        N_{pq} = \underbrace{N_{p-1,q}}_{\mathbf A} + \underbrace{N_{p,q-1}}_{\mathbf B} - \underbrace{N_{p-1, q-1}}_{\mathbf C} + \; t \cdot \underbrace{N_{p-1,q-1}}_{\mathbf D} + \; t \cdot \underbrace{|p-q| \cdot W_{pq}}_{\mathbf E},
    \end{equation}
    with initial values $N_{0,q} = N_{p,0} = N_{1,1} = 0$.
    
    We begin by verifying the initial values.
    Clearly when $p$ or $q$ is zero, the sum in~\eqref{Npq in theorem} is empty.
    If $p=q=1$, then there is a single summand where $k=1$, namely
    \[
        \Stri(1,0 \mid 1,0) \cdot t.
    \]
    Since $\Par(1 \times 0)$ contains only the empty partition, the expression above is zero, as required.

    It remains to verify the recursion itself.
    As our induction hypothesis, assume that~\eqref{Npq in theorem} is valid for all $N_{p'q'}$ with $p' < p$ or $q' < q$.
    It suffices to show, for all $k$, that the coefficient of $t^k$ in the right-hand side of~\eqref{Npq in theorem}, namely $\Stri(k, p-k \mid k, q-k)$, equals the coefficient of $t^k$ in the right-hand side of~\eqref{Npq expanded}.
    That is, writing $[t^k] f$ to denote the coefficient of $t^k$ in a polyomial $f(t)$,
    we need to show that
    \begin{align}
        \Stri(k, p-k \mid k, q-k) &= [t^k](\mathbf{A} + \mathbf{B} - \mathbf{C} + t\mathbf{D} + t\mathbf{E}) \nonumber \\
        &= [t^k]\mathbf{A} + [t^k]\mathbf{B} - [t^k]\mathbf{C} + [t^{k-1}]\mathbf{D} + [t^{k-1}]\mathbf{E}. \label{WTS} 
    \end{align}
    For ease of notation (and to align with Lemma~\ref{lemma:construction}), set $\ell \coloneqq p-k$ and $m \coloneqq q-k$.
    The required coefficients in $\mathbf{A}, \ldots, \mathbf{D}$ follow directly from~\eqref{Npq in theorem} via the induction hypothesis:
    \begin{alignat*}{3}
        [t^k]\mathbf{A} &= [t^k]N_{p-1,q} &&=\Stri(k, \ell-1 \mid k, m),\\
        [t^k]\mathbf{B} &= [t^k]N_{p,q-1} &&= \Stri(k, \ell \mid k, m-1),\\
        [t^k]\mathbf{C} &= [t^k]N_{p-1,q-1} &&= \Stri(k, \ell-1 \mid k, m-1),\\
        [t^{k-1}]\mathbf{D} &= [t^{k-1}]N_{p-1,q-1} &&= \Stri(k-1, \ell \mid k-1, m).
    \end{alignat*}
    The coefficient $[t^{k-1}]\mathbf{E}$ can be expressed by means of previous identities:
    \begin{align*}
        [t^{k-1}]\mathbf{E} &= |p-q| \cdot [t^{k-1}]W_{pq} \\
        &= |p-q| \cdot \textstyle\binom{p-1}{k-1} \binom{q-1}{k-1} && \text{by~\eqref{Wpq}}\\
        &= |(\ell+k) - (m+k)| \cdot \textstyle\binom{\ell + k-1}{k-1} \binom{m + k-1}{k-1}\\
        &= |\ell - m| \cdot |\Par((k-1) \times \ell)| \cdot |\Par((k-1) \times m)| && \text{by~\eqref{size Par ab}}.
    \end{align*}
  Upon substituting these five coefficients back into~\eqref{WTS}, we obtain the equation in Lemma~\ref{lemma:construction}, which verifies~\eqref{WTS} as desired. 
  Finally, to determine the range of the sum in~\eqref{Npq in theorem}, we observe that $\Stri(k, p-k \mid k, q-k)$ is nonzero only when $1 \leq k \leq \min\{p,q\}$.
\end{proof}

The following corollary establishes the palindromic part of Theorem~\ref{thm:main conjecture} (thus settling half of the original conjecture from~\cite{BW}).

\begin{corollary}
\label{cor:main}
    For all positive integers $n$, the polynomial $N_{n}(t)$ is palindromic. 
\end{corollary}

\begin{proof}
    Since $p=q=n$, the coefficient of $t^k$ in~\eqref{Npq in theorem} equals $\Stri(k,0 \mid k,0) = 0$; therefore $N_n(t)$ has degree $n-1$, and total degree $1 + (n-1) = n$.
    By the definition of palindromicity, we need to show that 
    \[
    \Stri(k, n-k \mid k, n-k) =  \Stri(n-k, k \mid n-k, k)
    \]
    for all $1 \leq k \leq n-1$, and this follows directly from Lemma~\ref{lemma:conjugate}.
\end{proof}

\begin{ex}

Recall from Section~\ref{sub:Npq recursion} the case $n=4$:
\[
N_4(t) = 20t + 56t^2 + 20t^3.
\]
By Theorem~\ref{thm:main result}, the coefficients are sums of sizes of symmetric differences of ordered pairs of partitions, where the partitions range over $\Par(k \times (4-k))$, for $k=1,2,3$.
The three tables below give the values $|\lambda \smtriangle \mu|$, taken over all ordered pairs $(\lambda,\mu)$ of partitions with the requisite shapes:
\[
\ytableausetup{boxsize=4pt}
\overset{\substack{\textstyle \Par(1 \times 3) \\ \phantom{.}}}{\underbrace{\begin{array}[t]{l|cccc}
     & \varnothing & \ydiagram{1} & \ydiagram{2} & \ydiagram{3} \\
     \hline
    \varnothing & 0 & 1 & 2 & 3 \\
    \ydiagram{1}
& 1 & 0 & 1 & 2\\
\ydiagram{2} & 2 & 1 & 0 & 1\\
\ydiagram{3} & 3 & 2 & 1 & 0
\end{array}}_{\text{sum = $20$}}}
\qquad\qquad
\overset{\substack{\textstyle \Par(2 \times 2) \\ \phantom{.}}}{\underbrace{\begin{array}[t]{l|cccccc}
     & \varnothing & \ydiagram{1} & \ydiagram{2} & \ydiagram{1,1} & \ydiagram{2,1} & \ydiagram{2,2}\\
     \hline
    \varnothing & 0 & 1 & 2 & 2 & 3 & 4 \\
    \ydiagram{1}
& 1 & 0 & 1 & 1 & 2 & 3\\
\ydiagram{2} & 2 & 1 & 0 & 2 & 1 & 2\\
\ydiagram{1,1} & 2 & 1 & 2 & 0 & 1 & 2 \\
\ydiagram{2,1} & 3 & 2 & 1 & 1 & 0 & 1\\
\ydiagram{2,2} & 4 & 3 & 2 & 2 & 1 & 0
\end{array}}_{\text{sum = $56$}}}
\qquad\qquad
\overset{\substack{\textstyle \Par(3 \times 1) \\ \phantom{.}}}{\underbrace{\begin{array}[t]{l|cccc}
     & \varnothing & \ydiagram{1} & \ydiagram{1,1} & \raisebox{1pt}{\ydiagram{1,1,1}} \\
     \hline
    \varnothing & 0 & 1 & 2 & 3 \\
    \ydiagram{1}
& 1 & 0 & 1 & 2\\
\ydiagram{1,1} & 2 & 1 & 0 & 1\\
\raisebox{-4pt}{\ydiagram{1,1,1}} & 3 & 2 & 1 & 0
\end{array}}_{\text{sum = $20$}}}
\]

\end{ex}

\section{Explicit formulas via minuscule lattices of Type A}
\label{sec:Defant}

Armed with the combinatorial interpretation in Theorem~\ref{thm:main result}, we show in this section that a recent result by Defant et al.~\cite{Defant}, concerning the Wiener index of minuscule lattices, yields the explicit formula for the coefficients of $N_n(t)$.
We use this to prove the unimodality of $N_n(t)$.
The connection goes even further.
By reinterpreting the results of~\cite{Defant}, we are now able to close the recursion in~\cite{BW} for the expected value of the one-dimensional EMD.

\subsection{Wiener index of minuscule lattices of Type A}

The minuscule lattices form a special class of Hasse diagrams, in which the underlying poset is the set of weights of a minuscule representation of a complex simple Lie algebra $\mathfrak g$.
(In this note we appeal only to the combinatorial structure of certain of these lattices, and therefore the following paragraph may be omitted without loss of continuity.)

A \emph{minuscule weight} of a complex simple Lie algebra $\mathfrak{g}$ is a minimal element $\varpi$ in the poset of dominant integral weights of $\mathfrak{g}$ (in the partial order whereby $\varpi \leq \varpi'$ if $\varpi' - \varpi$ is a nonnegative sum of simple roots).
For $\varpi$ minuscule, and $V_\varpi$ the irreducible representation of $\mathfrak{g}$ with highest weight $\varpi$, the 
set of weights of $V_{\varpi}$ is precisely the Weyl group orbit of $\varpi$, and its Hasse diagram is said to be a \emph{minuscule lattice}; see~\cite{Proctor} for details.
Minuscule lattices have various incarnations in the theory of Hermitian symmetric pairs $(\mathfrak{g}, \mathfrak{k})$.
For example, minuscule lattices are the Hasse diagrams (with respect to the Bruhat order) of the minimal-length representatives of the right cosets of the Weyl group of $\mathfrak{k}$ inside the Weyl group of $\mathfrak{g}$.
Equivalently, minuscule lattices are Hasse diagrams of the poset of lower-order ideals of the positive noncompact roots of $\mathfrak g$.
See~\cite{EHP} for an expansive exposition on the subject.

Let $G = (V,E)$ be a finite connected graph, and let $d(x,y)$ denote the distance in $G$ between $x,y \in V$.
The \emph{Wiener index} of $G$ is defined to be the sum of the distances between all ordered pairs of vertices:
\[
d(G) \coloneqq \sum_{\mathclap{(x,y) \in V \times V}} d(x,y).
\]
By extension, if $P$ is a finite poset, then $d(P)$ is defined to be the Wiener index of the Hasse diagram of $P$.

When $\mathfrak{g}$ is of Type A (in the Killing--Cartan classification), a minuscule lattice is the Hasse diagram of the poset of order ideals in a rectangle, say of dimensions $a \times b$, ordered by inclusion.
In~\cite{Defant}, this poset is denoted by $P_{a,b}$.
The key fact, for our purposes, is that $P_{a,b} \cong \Par(a \times b)$ as posets, if we order Young diagrams by inclusion.
In other words, in $\Par(a \times b)$, we declare $\mu \leq \lambda$ if $\lambda$ is obtained by adding boxes to $\mu$.
Hence $\varnothing$ is the minimal element of $\Par(a \times b)$, while $(b^a)$ is the maximal element.\footnote{This fact can be restated in somewhat different, but quite standard, combinatorial language:
if $\mathcal{Y}$ denotes \emph{Young's lattice} (the Hasse diagram of all Young diagrams ordered by inclusion), then we can identify $P_{a,b} \cong \Par(a \times b)$ with the lower order ideal of $\mathcal{Y}$ generated by the rectangular Young diagram with $a$ rows and $b$ columns.}
It is easy to see (and is noted in~\cite{Defant}*{\S1.3}) that distance on the Hasse diagram of $P_{a,b}$ measures the symmetric difference of the corresponding Young diagrams;
after all, neighboring elements in the Hasse diagram differ by exactly one box.
Therefore, given $\lambda,\mu \in \Par(a \times b)$, we have $d(\lambda,\mu) = |\lambda \smtriangle \mu|$.

The upshot of this is that the Wiener index of $\Par(a \times b)$ equals the sum of the sizes of the symmetric differences of all ordered pairs.
Translating between~\cite{Defant} and our present note, this means that
\begin{equation}
    \label{Wiener equals S tri}
    d(P_{a,b}) = \Stri(a,b) \coloneqq \Stri(a,b \mid a,b).
\end{equation}
(From now on we will need only the specialized shorthand $\Stri(a, b)$, since we consider ordered pairs of elements from the common set $\Par(a \times b)$.)
By manipulating the generating functions of certain Motzkin paths, the authors of~\cite{Defant} derive the following formula (Thm.~1.2) for the Wiener index of $P_{a,b}$:
\begin{equation}
    \label{Wiener index of Pab}
    d(P_{a,b}) = \frac{ab}{4a + 4b + 2} \binom{2a+2b+2}{2a+1}.
\end{equation}
Combining this result with~\eqref{Wiener equals S tri} and our Theorem~\ref{thm:main result}, we obtain the following explicit formula for the coefficients of $N_n(t)$:

\begin{theorem}
\label{thm:coeffs of Nn}
    For all positive integers $n$, we have
    \[
    N_n(t) = \frac{1}{4n + 2} \cdot \sum_{k=1}^{n-1} k(n-k) \binom{2n+2}{2k+1} \: t^k.
    \]
\end{theorem}

We now prove the unimodal part of Theorem~\ref{thm:main conjecture}, thus settling the original conjecture in~\cite{BW}:

\begin{corollary}
\label{cor:unimodal}
    For all positive integers $n$, the polynomial $N_n(t)$ is unimodal.
\end{corollary}

\begin{proof}
    Using Theorem~\ref{thm:coeffs of Nn}, it suffices to show that both factors $k(n-k)$ and $\binom{2n+2}{2k+1}$ are nondecreasing as functions of $k$, for $1 \leq k < \lfloor n/2 \rfloor$.
    In fact, both factors are strictly increasing.
    In particular, it is well known that the binomial coefficients $\binom{a}{b}$ are increasing as functions of $b$, for $b < a/2$, and therefore $\binom{2(n+1)}{2k+1}$ is increasing for $k < n/2$.
    Moreover, treating $k(n-k)$ as a real function of $k$, we have $\frac{d}{dk} [k(n-k)] = n-2k$, which is positive for all $k < n/2$.
    This completes the proof.
\end{proof}

\subsection{Expected value of the EMD}

Not only does the formula~\eqref{Wiener index of Pab} lead to the complete description of $N_n(t)$ in Theorem~\ref{thm:coeffs of Nn}, but we now show that it also solves the recursion for the expected value of the EMD, which was the main result of~\cite{BW}.

The paper~\cite{Erickson23}*{eqn.~(10)} employs the bijection
\begin{align*}
    \C(s,n) &\longrightarrow \Par(s \times (n-1)),\\
    \alpha & \longmapsto \Big((n-1)^{\alpha_1}, (n-2)^{\alpha_2}, \ldots, 2^{\alpha_{n-2}}, 1^{\alpha_{n-1}}\Big),
\end{align*}
where the exponents denote repeated parts in a partition.
Now let $\alpha,\beta \in \C(s,n)$, and suppose that $\alpha \mapsto \lambda$ and $\beta \mapsto \mu$ under this bijection.
It is shown in~\cite{Erickson23}*{Prop.~3.1} that the EMD is precisely the size of the symmetric difference of the corresponding Young diagrams:
\[
    \EMD(\alpha,\beta) = |\lambda \smtriangle \mu|.
\]
Comparing this with~\eqref{H prime} and~\eqref{S triangle}, it follows that 
\begin{equation}
\label{sum EMD equals Stri}
    [t^s] H'_n(t) = \sum_{(\alpha,\beta)} \EMD(\alpha,\beta) = \sum_{(\lambda,\mu)}  |\lambda \smtriangle \mu| \eqqcolon \Stri(s, n-1),
\end{equation}
where $\alpha,\beta \in \C(s,n)$ and $\lambda,\mu \in \Par(s \times (n-1))$.
Therefore, the expected value of the EMD is simply $\Stri(s,n-1)$ divided by the square of $|\C(s,n)|$, or equivalently, divided by the square of $|\Par(s \times (n-1))|$.

\begin{theorem}
    \label{thm:EMD exp val}
    Let $(\alpha, \beta) \in \C(s,n) \times \C(s,n)$ be chosen uniformly at random.
    Then
    \[
    \mathbb{E}\Big[\EMD(\alpha,\beta)\Big] = \frac{s(n-1)}{4s+4n-2} \cdot \frac{\binom{2s+2n}{2s+1}}{\binom{s+n-1}{s}^2}.
    \]
\end{theorem}

\begin{proof}
    We have 
    \begin{align*}
    \mathbb{E}\Big[\EMD(\alpha,\beta)\Big] &\coloneqq \sum_{(\alpha,\beta)} \EMD(\alpha,\beta) \Big/ |\C(s,n)|^2 \\
    &= \Stri(s, n-1) \Big/ \textstyle \binom{s+n-1}{s}^2 && \text{by~\eqref{sum EMD equals Stri} and~\eqref{size Par ab}}\\
    &= \frac{s(n-1)}{4s+4(n-1) + 2} \cdot \textstyle \binom{2s + 2(n-1)+2}{2s+1}\Big/ \binom{s+n-1}{s}^2 && \text{by~\eqref{Wiener equals S tri} and \eqref{Wiener index of Pab},}
    \end{align*}
    which yields the expression in the theorem.
\end{proof}

\begin{rem}
\label{rem FV}
    The explicit formula in Theorem~\ref{thm:EMD exp val} gives us an easy way to compute the expected value of the EMD on ordered pairs of probability distributions on $[n]$:
    namely, divide by $s$ so that $\alpha/s$ and $\beta/s$ are probability distributions with rational values, and then take the limit as $s \rightarrow \infty$.
    This was the method used in~\cite{BW} to convert their recursive formula from compositions to probability distributions.
    In this case, starting with the expression in Theorem~\ref{thm:EMD exp val}, it can be shown that
    \[
    \lim_{s \rightarrow \infty} \frac{1}{s} \cdot \mathbb{E}\Big[\EMD(\alpha,\beta)\Big] = \frac{\sqrt{\pi} \, (n-1) \,\Gamma(n)}{4 \, \Gamma\!\!\left(n + \frac{1}{2}\right)}.
    \]
    This is, although somewhat disguised, equivalent to the expected value formula derived in~\cite{FV}*{Thm.~1}, namely
    \[
    \frac{2^{2n-3}(n-1)}{(2n-1)!}\,(n-1)! \, ^2,
    \]
    which was obtained analytically  by solving the recursion in~\cite{BW}.
\end{rem}

\section{Further observations regarding $N_n(t)$}
\label{sec:problems}

The following symmetric difference identity is proved in~\cite{LaHaye}*{eqn.~(26)}:
\[
    S(n) \coloneqq \sum_{\mathclap{\substack{(X,Y) \\ \in \: 2^{[n]} \times 2^{[n]}}}} | X \smtriangle Y| = n \cdot 2^{2n-1},
\]
where as usual $2^{[n]}$ denotes the power set of $[n]$.
This is the sequence A002699 in the OEIS, which also gives an alternative formula $S(n) = \sum_{k=1}^n k \binom{2n}{k}$.
The following fact has now been proved by~\cite{Ding} since our conjecture in an early draft of this note;
we remain interested, however, in finding a bijective proof of some kind, since it seems that the symmetric differences on each side can hardly be coincidental.

\begin{prop}
\label{conj N equals S}
    For all positive integers $n$, we have $N_n(1) = S(n-1)$.
\end{prop}

Real-rooted polynomials have garnered significant interest in combinatorics.
Famous examples include the Eulerian polynomials (which are the generating polynomials of the descent statistic on the symmetric group, or more generally on multiset permutations), graph matching polynomials, rook polynomials, and polynomials with interlacing or interweaving roots; we recommend the excellent survey~\cite{Branden}.
The following fact has also been proved in~\cite{Ding} since we conjectured it:

\begin{prop}
    For all positive integers $n$, the polynomial $N_n(t)$ has only real roots.
\end{prop}

As pointed out by a referee, rational functions with palindromic numerators play an important role in commutative algebra, in relation to Gorenstein rings (roughly speaking, rings that are in some sense ``self-dual'').
The key result~\cite{Stanley78}*{Thm.~4.4} is that a Cohen--Macaulay ring is a Gorenstein ring if and only if the numerator of its reduced Hilbert series is palindromic.
It would be an interesting problem to realize the rational function $H'_n(t)$, as defined in~\eqref{H prime}, as the Hilbert series of some Gorenstein ring.

 \bigskip

 \noindent \scriptsize Declarations of interest: none.

\bibliographystyle{alpha}
\bibliography{references}

\end{document}

%% file: S_def.tex
\Stri\!\left(
\;\begin{tikzpicture}[scale=.6,baseline={([yshift=-\the\dimexpr\fontdimen22\textfont2\relax]
                    current bounding box.center)},thick]
\draw[fill=gray!50] (0,0) rectangle (1.5,1);
\end{tikzpicture} \hspace{1ex}
\begin{tikzpicture}[scale=.6,baseline={([yshift=-\the\dimexpr\fontdimen22\textfont2\relax]
                    current bounding box.center)},thick]
\draw[fill=gray!50] (0,0) rectangle (2,1);
\end{tikzpicture} 
\;
\right)

%% file: L_def.tex
\begin{tikzpicture}[scale=.6,baseline=2mm,thick]
\draw[gray,dashed] (0,0) rectangle (1.5,1);
\draw[gray,dashed] (1.5,1) -- (2.2,1) -- (2.2,0) -- (1.5,0);
\draw[{Latex[scale=.7]}-{Latex[scale=.7]}] (1.5,.5) -- (2.2,.5);
\node[label=right:{$\coloneqq |\ell - m|,$}] at (2.2,.5) {};
\end{tikzpicture}

%% file: S_LHS.tex
\Stri\!\left(
\;\begin{tikzpicture}[scale=.6,baseline={([yshift=-\the\dimexpr\fontdimen22\textfont2\relax]
                    current bounding box.center)},thick]
\draw[fill=gray!50] (0,0) rectangle (1.5,1);
\end{tikzpicture} \hspace{1ex}
\begin{tikzpicture}[scale=.6,baseline={([yshift=-\the\dimexpr\fontdimen22\textfont2\relax]
                    current bounding box.center)},thick]
\draw[fill=gray!50] (0,0) rectangle (2,1);
\end{tikzpicture} 
\;
\right)

%% file: S_term1.tex
\Stri\!\left(
\;\begin{tikzpicture}[scale=.6,baseline={([yshift=-\the\dimexpr\fontdimen22\textfont2\relax]
                    current bounding box.center)},thick]
\draw[fill=gray!50] (0,0) rectangle (1.3,1);
\draw (0,0) rectangle (1.5,1);
\end{tikzpicture} \hspace{1ex}
\begin{tikzpicture}[scale=.6,baseline={([yshift=-\the\dimexpr\fontdimen22\textfont2\relax]
                    current bounding box.center)},thick]
\draw[fill=gray!50] (0,0) rectangle (2,1);
\end{tikzpicture} 
\;
\right)

%% file: S_term2.tex
\Stri\!\left(
\;\begin{tikzpicture}[scale=.6,baseline={([yshift=-\the\dimexpr\fontdimen22\textfont2\relax]
                    current bounding box.center)},thick]
\draw[fill=gray!50] (0,0) rectangle (1.5,1);
\end{tikzpicture} \hspace{1ex}
\begin{tikzpicture}[scale=.6,baseline={([yshift=-\the\dimexpr\fontdimen22\textfont2\relax]
                    current bounding box.center)},thick]
\draw[fill=gray!50] (0,0) rectangle (1.8,1);
\draw (0,0) rectangle (2,1);
\end{tikzpicture} 
\;
\right)

%% file: S_term3.tex
\Stri\!\left(
\;\begin{tikzpicture}[scale=.6,baseline={([yshift=-\the\dimexpr\fontdimen22\textfont2\relax]
                    current bounding box.center)},thick]
\draw[fill=gray!50] (0,0) rectangle (1.3,1);
\draw (0,0) rectangle (1.5,1);
\end{tikzpicture} \hspace{1ex}
\begin{tikzpicture}[scale=.6,baseline={([yshift=-\the\dimexpr\fontdimen22\textfont2\relax]
                    current bounding box.center)},thick]
\draw[fill=gray!50] (0,0) rectangle (1.8,1);
\draw (0,0) rectangle (2,1);
\end{tikzpicture} 
\;
\right)

%% file: S_term4.tex
\Stri\!\left(
\;\begin{tikzpicture}[scale=.6,baseline={([yshift=-\the\dimexpr\fontdimen22\textfont2\relax]
                    current bounding box.center)},thick]
\draw[fill=gray!50] (0,0) rectangle (1.5,0.8);
\draw (0,0) rectangle (1.5,1);
\end{tikzpicture} \hspace{1ex}
\begin{tikzpicture}[scale=.6,baseline={([yshift=-\the\dimexpr\fontdimen22\textfont2\relax]
                    current bounding box.center)},thick]
\draw[fill=gray!50] (0,0) rectangle (2,0.8);
\draw (0,0) rectangle (2,1);
\end{tikzpicture} 
\;
\right)

%% file: L_term5.tex
\begin{tikzpicture}[scale=.6,baseline={([yshift=-\the\dimexpr\fontdimen22\textfont2\relax]
                    current bounding box.center)},thick]
\draw[gray,dashed] (0,0) rectangle (1.5,1);
\draw[gray,dashed] (1.5,1) -- (2.2,1) -- (2.2,0) -- (1.5,0);
\draw[{Latex[scale=.7]}-{Latex[scale=.7]}] (1.5,.5) -- (2.2,.5);
\end{tikzpicture}

%% file: Par1_term5.tex
\#\Par \! \left(
\;\begin{tikzpicture}[scale=.6,baseline={([yshift=-\the\dimexpr\fontdimen22\textfont2\relax]
                    current bounding box.center)},thick]
\draw[fill=gray!50] (0,0) rectangle (1.5,.8);
\draw (0,0) rectangle (1.5,1);
\end{tikzpicture}
\;
\right)

%% file: Par2_term5.tex
\#\Par \! \left(
\;\begin{tikzpicture}[scale=.6,baseline={([yshift=-\the\dimexpr\fontdimen22\textfont2\relax]
                    current bounding box.center)},thick]
\draw[fill=gray!50] (0,0) rectangle (2,.8);
\draw (0,0) rectangle (2,1);
\end{tikzpicture}
\;
\right)